\documentclass[a4paper, 12pt]{article}

\usepackage[utf8]{inputenc}
\usepackage[english]{babel}
\usepackage[top=2cm,bottom=2cm,left=2cm,right=2cm,marginparwidth=1.75cm]{geometry}
\usepackage{amsmath}
\usepackage{amsthm}
\usepackage{amsfonts}
\usepackage{amssymb}
\usepackage{upgreek}
\usepackage{mathtools}

\usepackage{graphicx}
\usepackage{indentfirst}
\usepackage{bbm}
\usepackage{csquotes}
\usepackage[colorlinks=true, allcolors=blue]{hyperref}
\usepackage{calrsfs}
\DeclareMathAlphabet{\pazocal}{OMS}{zplm}{m}{n}

\newcommand{\defeq}{\mathrel{\mathop:}=}
\newcommand{\eqdef}{\mathrel{\mathop=}:}

\newcommand{\mapto}{ \xrightarrow[]{}}

\newcommand{\eqcoord}{\dot{=}}
\newcommand{\hrz}{\mathsf{h}} 

\newcommand{\dR}{\mathbb{R}}
\newcommand{\Prob}{\mathbb{P}}
\newcommand{\dE}{\mathbb{E}}
\newcommand{\FiltProbSpBrwn}{((\Omega, \mathcal{F}, \Prob), W_\cdot)}

\newcommand{\indic}[1]{\mathbbm{1}_{\{ #1\}}}

\newtheorem{theorem}{Theorem}[section]
\newtheorem{proposition}[theorem]{Proposition}

\theoremstyle{remark}
\newtheorem{remark}[theorem]{Remark}

\theoremstyle{definition}

\theoremstyle{definition}
\newtheorem{definition}[theorem]{Definition}

\theoremstyle{definition}

\newtheorem{lemma}[theorem]{Lemma}

\title{A practical global existence and uniqueness result for stochastic differential equations on Riemannian manifolds of bounded geometry}
\author{Matthias Rakotomalala\thanks{CMAP, CNRS, École polytechnique, Institut Polytechnique de Paris, 91120 Palaiseau, France}}
\date{}

\begin{document}
\maketitle

\begin{abstract}
In this paper, we establish a result for existence and uniqueness of stochastic differential equations on Riemannian manifolds, for regular inhomogeneous tensor coefficients with stochastic drift, under geometrical hypothesis on the manifold, so-called manifolds of bounded geometry. Furthermore, we provide stochastic flow estimates for the solutions.
\end{abstract}

\tableofcontents

\section{Introduction}
\label{sec:Intro}

In this paper, we establish an existence and uniqueness result for stochastic differential equations on Riemannian manifolds with regular inhomogeneous tensor coefficients and a stochastic drift. We work under the assumption that the manifold has bounded geometry. Our main motivation is to obtain a global existence result suitable for applications and modeling purposes. In particular, we are interested in stochastic control problems, where global existence in time and the stochastic nature of the drift enable to consider open-loop control problems. To the best of our knowledge, this specific result does not appear in the literature, and we believe it could support the analysis of mathematical models that rely on stochastic differential equations on manifolds.

From the pioneering work of K. Itô~\cite{ito1950}, there has been significant development in the literature on stochastic differential equations on manifolds. Elworthy~\cite[Theorem 4, Chapter VIII]{elworthy1982stochastic} established a global existence and uniqueness result for a general deterministic homogeneous operator, assuming the manifold is compact. In the book by Ikeda and Watanabe~\cite[Theorem V.4.2]{ikeda2014stochastic}, an existence and uniqueness result up to a blow-up time is provided for time-homogeneous deterministic coefficients. The comparison theorem, in this context, allows for non-blow-up conditions for the minimal Brownian motion on a general manifold, given lower bounds on the Ricci curvature tensor, as shown in~\cite[Corollary VI.5.2, p. 466]{ikeda2014stochastic} and~\cite{hsu2002stochastic}. Wang~\cite[Corollary 2.1.2]{wang2014analysis} proved that, under an additional growth condition on the drift, the drifted Brownian motion does not blow up.

The references above are concerned with criteria on the geometry for the drifted or minimal Brownian motion to be non-explosive, namely, the stochastic process generated by $\frac{1}{2}\Delta + B \cdot \nabla$,
where $B$ is a vector field and $\Delta$ is the Laplace-Beltrami operator.
In contrast, we establish sufficient geometric conditions under which the stochastic process generated by  
$\frac{1}{2} \Sigma \cdot \nabla^2 + \tilde{B} \cdot \nabla$ is non-explosive, where $\Sigma = A \cdot A^*$ and $\tilde{B} = B + \frac{1}{2} \nabla A \cdot A$, with $B$ being a bounded regular vector field and $A$ a bounded regular $(1,1)$-tensor field. Our main result accommodates time-inhomogeneous coefficients $A$ and $B$, as well as stochasticity in $B$. As a particular case, when $A = \text{Id}_{TM}$ and $B$ is deterministic, we recover the same process as in the references above.

In order to ensure the non-explosion of the process, we assume that the manifold has bounded geometry—namely, completeness, a positive injectivity radius, and a bounded Riemann curvature tensor. We apply Itô’s switching strategy~\cite{ito1950} to derive local uniform estimates for the process in charts and establish the non-blow-up property while constructing the solution.

In the analysis of partial differential equations, manifolds of bounded geometry are widely used as they provide a suitable framework for studying non-compact manifolds. This framework allows for the construction of a regular partition of unity and control over the asymptotic behavior of the geometry at infinity. Examples of manifolds of bounded geometry include compact manifolds, Euclidean space, hyperbolic space in dimension $d$, and any Cartesian product of manifolds of bounded geometry. Manifolds of bounded geometry also possess smooth approximate distance functions. In Section~\ref{sec:MainResults}, this property is used to control the growth at infinity of the process.

In Section~\ref{sec:Intro}, we recall results from Riemannian geometry and specify the notion of solution considered here. This notion coincides with the concept of a lifted solution, as used in~\cite{ikeda2014stochastic} and~\cite{elworthy1982stochastic}. In Section~\ref{sec:MainResults}, we construct the solution using a regular partition of unity and directly ensure the non-blow-up property. We then provide integrability estimates for the stochastic flow.

\textbf{Notations} :
In this paper, we use the following notations: $|\cdot|$ denotes the absolute value, and $\| \cdot \|$ denotes the Euclidean norm. If $\phi$ is a local chart, we set $\phi^{-1} = \psi$, so that $\psi^\alpha$ is the inverse of $\phi^\alpha$, where $\alpha$ is the index of the chart. We use capital letters for tensors, lowercase letters for their associated coefficients in local charts, and adopt the Einstein summation convention. The symbol $\eqcoord$ indicates that the right-hand side of the equation is written in coordinates. The manifold considered here is assumed to be Riemannian; we write $g_{ij}$ (respectively $g^{ij}$) for the components of the metric tensor (resp. its inverse) in a local coordinate system. We assume that the manifold is equipped with its Levi-Civita connection, and we write $\nabla_{\partial_i} \partial_j \eqcoord \Gamma^k_{ij}$, the Christoffel symbol, which, in the case of the Levi-Civita connection, is given by $\Gamma^k_{ij} = \frac{1}{2} g^{kl} (\partial_i g_{jl} + \partial_j g_{il} - \partial_l g_{ij})$. Let $T^n_m M$ be the $(n,m)$-tensor bundle over $M$, consisting of tensors that are contravariant of order $n$ and covariant of order $m$. Then, $C^k(T^n_m M)$ denotes the space of $k$-times differentiable sections of $T^n_m M$, and $\nabla^k A$ is the $k$-th order covariant derivative of a tensor $A$. We denote by $\| \cdot \|_g$ the norm induced by $g$ on $C(T^n_m M)$, with $\| A \|_g = \sup_{x \in M} |A|_g$. The notation $A \cdot B$ indicates the scalar product of the vector fields $A$ and $B$. The notation $\circ$ denotes the Stratonovich integral.
\\

\subsection{Riemannian Geometry}
In this section, we recall some results from Riemannian geometry.

\begin{definition}[Smooth Manifold]
A d-dimensional \textit{manifold} $M$ is a second-countable Hausdorff space, such that there exists an atlas $\{(O^\alpha, \phi^\alpha)\}_{\alpha\in \mathfrak{K}}$, where the $O^\alpha$ forms an open cover of $M$ and $\phi^\alpha$ is an \textit{homeomorphism} from $M$ to an open subset of $\dR^d$, and $\phi^\alpha \circ \psi^\beta$ is smooth whenever it makes sense. A couple $(O, \phi)$ is called a local chart.
A \textit{Riemannian metric} $g$ on $M$ assigns to each point $x\in M$ a positive-definite inner product on the tangent space at $x$, 
\begin{equation*}
    g_{x}:T_{x}M\times T_{x}M \to \dR,
\end{equation*}
in a smooth way. A \textit{Riemannian manifold} $(M,g)$ is a smooth manifold $M$ equipped with a Riemannian metric $g$.
\end{definition}

In the following, $(M,g)$ is a Riemannian manifold equipped with its Levi-Civita connection. We use Christoffel symbols $\Gamma^l_{jk}$ to denote the connection coefficients in local charts.

\begin{definition}[Riemann Curvature Tensor]
    Let $X,Y,Z\in C^2(TM)$ be three vector fields on $M$. We define the \textit{Riemann curvature tensor} with the following formula,
    \begin{equation*}
        R(X,Y)Z = \nabla_X\nabla_Y Z - \nabla_Y\nabla_X Z - \nabla_{[X,Y]}Z,
    \end{equation*}
    where $[X,Y]$ is the Lie bracket of the vector fields.
    In local coordinates, this can be expressed as
    \begin{equation*}
        R^l_{ijk} = \partial_j\Gamma^l_{ik} - \partial_k \Gamma^l_{ij} + \Gamma^l_{jm} \Gamma^m_{ik} - \Gamma^l_{km} \Gamma^m_{ij}.
    \end{equation*}
\end{definition}




\begin{definition}
    Let $i : M \to \dR_+$ be the \textit{injectivity radius} function, defined as $i(x)$ is the largest radius for which the exponential map $\exp_x : T_xM \to M$ is a diffeomorphism. The injectivity radius of $M$, noted $i(M)$, is defined as,
    \begin{equation*}
         i(M) = \inf_{m \in M} i(x).
    \end{equation*}
    A manifold is said to have a \textit{positive injectivity radius} if $i(M) > 0$.
\end{definition}

We now recall the definition of Manifold of bounded Geometry.

\begin{definition}[Manifold of bounded Geometry]
    A Riemannian manifold $(M,g)$ equipped with its Levi-Civita connection is said to be \textit{of bounded geometry} if: it is a \textit{complete} metric space, it has \textit{positive injectivity radius}, and the Riemann curvature tensor is bounded,
    \label{def:BoundedGeo}  
    \begin{equation*}
            \| R \|_g \leq C.
    \end{equation*}
\end{definition}
Examples of manifolds with bounded geometry include Euclidean space, the $d$-dimensional hyperbolic space, the $d$-dimensional sphere, any compact manifold, and any Cartesian product of manifolds of bounded geometry. In the absence of compactness, the bounded geometry condition ensures control over the manifold's global structure and analytical properties.

\begin{proposition}
    Suppose that $(M,g)$ is of bounded geometry, then there exists a constant $C_M>0$, such that for any $x\in M$ and for any normal charts around
    $x$ the following estimates holds in coordinates,
    \begin{equation*}
        \max_{i,j}|g_{ij}(x)|,\max_{i,j}|g^{ij}(x)|, \max_{i,j,k} |\Gamma^k_{ij}(x)| \leq C_M, \forall x\in B_{\dR^d}(0,i_M).
    \end{equation*}
\end{proposition}
This result is presented in~\cite{eichhorn1991boundedness}.

\subsection{Orthonormal frame bundle and horizontal lift}
Building on the previous notions of Riemannian geometry, we now introduce additional concepts required to define \textit{rolling without slipping} solutions. For a comprehensive presentation of the following concepts, we refer to \cite[chp 2, p.35]{hsu2002stochastic} and \cite[sec 7.2, p.127]{bishop2011geometry}.
\begin{definition}
    Let $OM$ denote the set of all orthonormal basis of the tangent space at each point of $M$,
    \begin{equation*}
        OM = \{ (x, E_1, \cdots, E_d) | x\in M, (E_1, \cdots, E_d) \text{ is an orthonormal basis of } T_xM  \},
    \end{equation*}
    and denote by $\pi : OM \mapto M$, the canonical projection. Then, there exists a manifold structure on $OM$, that makes $(OM,\pi)$ a principal bundle over $M$,
    called the \textit{orthonormal frame bundle}. An element $u \in OM$, is identifiable as an isometry $u : \dR^d \mapto T_{\pi(u)}M$, for an element $\lambda \in \dR^d$, we will note $u\lambda = \sum_{i = 1}^d E_i \lambda^i \in T_{\pi(u)}M$.
\end{definition}

It may seem natural to find a straightforward way to map the Gaussian noise onto the tangent space by selecting a global section of \( OM \). However, the existence of such a section would require that \( OM \) is globally trivializable. As a counterexample, the 2-sphere \( \mathbb{S}^2 \) does not satisfy this property. Therefore, we must specify how the frame is transported along the path of the process. For this purpose, we introduce the \textit{horizontal lift} of a vector field.

\begin{definition}[Horizontal Lift]
     A smooth curve $(u_t)$ taking values in $OM$, is said to be \textit{horizontal} if for each $e \in \dR^d$ the vector field $(u_te)$ is parallel along the $M$-valued curve $(\pi(u_t))$. A tangent vector $Y \in T_uOM$ is said to be horizontal if it is the tangent vector of a \textit{horizontal} curve at $u$. The space of horizontal vectors at $u$ is denoted by $H_uOM$, we have the decomposition
     \begin{equation*}
         T_uOM = V_uOM \oplus H_uOM,
     \end{equation*}
    where $V_uOM$ is the subspace of \textit{vertical vectors}, that are tangent to the fiber $T_uOM$. It follows that the canonical projection $\pi$, induces an isomorphism $\pi_\hrz : H_uOM\mapto T_{\pi(u)}M$, and for each $B \in T_xM$ and a frame $u$ at $x$, there is a unique horizontal vector $B^\hrz$, the \textit{horizontal lift} of $B$ to $u$, such that $\pi_\hrz(B^\hrz) = B$. Thus if $B$ is a vector field on $M$, then $B^\hrz$ is a vector field on $OM$.\\
    In coordinates $\{x^i,\zeta^j_k\}$, the lifted vector field can be expressed as,
    \begin{equation*}
         B^\hrz \eqcoord b^i(x)\frac{\partial}{\partial x^i} - \Gamma^k_{ij}(x)b^i(x)\zeta^j_m\frac{\partial}{\partial \zeta^k_m}.
    \end{equation*}
\end{definition}

\subsection{Moving frame strong solution}
In this section, we introduce the notion of solution considered in this paper. It coincides with the concept of \textit{rolling without slipping} or \textit{moving frame} solutions introduced by Eells and Elworthy in \cite{eells1976stochastic}. For a detailed presentation of this notion, we refer to \cite{elworthy1982stochastic,matthes1986ikeda,hsu2002stochastic}. The idea is to lift the equation to the \textit{orthonormal frame bundle}, ensuring chart-independent solutions by using an initial basis of the tangent space to map the noise onto the manifold.
The differential structure of a manifold is constructed from local trivializations. Therefore, a continuously adapted process on a manifold is a solution to the desired stochastic differential equation if it is a solution in local charts. To this end, we must introduce a \textit{local regularization} of the dynamical system, namely an atlas, where the integrals in local coordinates are well-defined.

\begin{definition}[Regular Localization]
    Given $((\Omega, \mathcal{F}, \Prob), W_\cdot)$, a filtered probability space with a $\Prob$-complete filtration, equipped with an adapted $d$-dimensional Brownian motion. Let $A : [0,T] \mapto C^2_b(T^1_1M), B : [0,T]\times \Omega \mapto C_b(T^1M)$, such that $B$ is progressively measurable. 
    We say that an atlas $\{ (O^\alpha, \phi^\alpha) \eqdef \Lambda^\alpha, \alpha \in \mathfrak{K}\}$, is a \textit{local regularization} of the stochastic differential system associated to $(A,B)$, if in any chart $\alpha \in \mathfrak{K}, \exists\ C_{\alpha} > 0 $, such that 
    \begin{align*}
        \sup_{t \in [0,T], x\in \phi^\alpha(O^\alpha)} \max_{i =1}^d |b^i (t, \omega, x)| < C_{\alpha} , \Prob-a.s,\\
        \sup_{t \in [0,T], x\in \phi^\alpha(O^\alpha)} \max_{i,j,k=1}^d |\partial^n_k a^i_j (t, x)| < C_{\alpha}, \text{ for } n = 0,1, \\
        \sup_{x\in \phi^\alpha(O^\alpha)}\max_{k,i,j = 1}^{d}|\Gamma^k_{ij}(x))| < C_{\alpha}, \sup_{x\in \phi^\alpha(O^\alpha)}\max_{i,j = 1}^{d} |g_{ij}(x)| < C_{\alpha},
    \end{align*}
    where $a^l_k(t, x)$, $b^i(t,\omega, x)$, $g_{ij}(x)$, $\Gamma^k_{ij}(x)$ are respectively the coefficients of the tensors $A$, $B$, the metric tensor $g$ and the Christoffel symbol in the coordinate chart $\alpha$.
\end{definition}

\begin{definition}
    \label{def:hStrongSol}
    Let $((\Omega, \mathcal{F}, \Prob), W_\cdot)$ be a filtered probability space with a $\Prob$-complete filtration, equipped with an adapted $d$-dimensional Brownian motion. Let $\{e_i\}_{i=1}^d$ be the coordinate unit vectors of $\dR^d$. 
    Let $A : [0,T] \mapto C^2_b(T^1_1M)$, and $B : [0,T]\times \Omega \mapto C_b(T^1M)$ a progressively measurable tensor valued process. For $u \in OM$ define, $A^\hrz_i(t, u) \defeq (A(t,\pi(u))ue_i)^\hrz \defeq (A_i(t, u))^\hrz, B^\hrz(t, \omega, u) =(B(t,\omega, \pi(u)))^\hrz$.
    A progressively measurable adapted and continuous process $(U_\cdot)_{[0,T]}$ taking values in $OM$, is said to be a \textit{moving frame strong solution} of
    \begin{equation*}
        \label{Eq:TOM}
        \tag{$E_{M}^\hrz$}
        dU_t = A^\hrz_i(t,U_t)\circ dW^i_t + B^\hrz(t,U_t)dt \text{ on } T_{U_t}OM, 
    \end{equation*}
    with $\mathcal{F}_0$-measurable initial value $U_0 \in OM$, if there exists, a \textit{regular localization}, that \textit{affirms} the equation. That is, if for any $ t \in [0,T], \alpha \in \mathfrak{K}$, defining, $\Omega^{\alpha}_t \defeq \{ \omega \in \Omega, \pi(U_t(\omega)) \in O^\alpha \}$, $\tau^{\alpha}_t : \Omega^{\alpha}_t \mapto [t,T]$ the exit time from $O^\alpha$, and $(\xi^i(t), \zeta^j_k(t)) = (\phi^\alpha(\pi(U_t)), d\phi^\alpha(\pi(U_t))U_t \partial/\partial x^k)$ the local trivialization of the process in the chart, with $\partial/\partial x^k$ the associated basis. We have that, $(\xi^i(t), \zeta^j_k(t))$ is $\Prob$-as on $\Omega^{\Lambda^\alpha}_t$ up to $\tau^\alpha_t$ solution of,
    \begin{align*}
        \label{Eq:InChart}
        \tag{$E_{\Lambda}$}
        \begin{cases}
            d\xi^i(t)  = (a^i_l(t,\xi(t))\zeta^l_m(t))\circ dW^m_t + b^i(t,\xi(t))dt,\\
            d\zeta^k_m(t) = (-\Gamma^k_{ij}(\xi(t))\zeta^i_m(t) a^j_l(t,\xi(t))\zeta^l_n(t))\circ dW^n_t + (-\Gamma^k_{ij}(\xi(t))b^i(t,\xi(t)) \zeta^j_m(t))dt,
        \end{cases}
    \end{align*}
   where $a^i_l,b^i$ and $\Gamma^k_{ij}$ are respectively the coefficients associated to the tensors $A,B$ and the Christoffel symbol, in the chart $\Lambda^\alpha$.
\end{definition}

All the hypotheses on the regular localization ensure that equation \eqref{Eq:InChart} is well-defined. Indeed all the coefficients involved are supposed to be bounded in the chart. And the $C^2$ regularity of the tensor $A$ allows to define the Stratonovich integral. Finally, since $U_t$ is an orthonormal basis of $T_{\pi(U_t)}M$, we have in coordinate that $\zeta^i_k g_{ij}(\xi) \zeta^j_k = 1$, which implies, since the coefficients of the metric tensor are bounded, that $|\zeta^i_k| < K$ for some $K$ possibly depending on the chart. 

\begin{remark}
    Since any element of $OM$ can be identified as a point $x$ in $M$ and an associated orthonormal basis of $T_xM$, equation \eqref{Eq:TOM} describes a moving frame along a process $X_t = \pi(U_t)$ taking values in $M$, itself solution of :
    \begin{equation*}
        dX_t = A_i(t,U_t)\circ dW^i_t + B(t,X_t)dt \text{ on } T_{X_t}M.
    \end{equation*}
    The equation on $X$ is not autonomous, as it requires accounting for the evolution of the basis of the tangent plane while \textit{rolling} the manifold over the noise.
\end{remark}

\begin{remark}
    The Stratonovich integral formulation is essential for obtaining a chart-independent definition of a solution. Without the chain rule provided by the Stratonovich integral, a second-order term would arise when changing coordinates, corresponding to the corrective term in the Itô formulation. This term would not align with the coordinate transformation of a tensor field, and as a result, the process would no longer be a solution to the same equation under a change of coordinates.
\end{remark}

\section{Main results}
\label{sec:MainResults}
\subsection{Existence and uniqueness result}

We now present the main result of this paper, with its proof provided at the end of this section.
\begin{theorem}[Existence and Uniqueness]
    \label{thm:ExistUniq}
    Let $(M,g)$ be a Riemannian Manifold of bounded geometry, and $\FiltProbSpBrwn$ a filtered probability space with a $\Prob$-complete filtration, equipped with an adapted $d$-dimensional Brownian motion. Let $A$ be a deterministic $(1,1)$-tensor valued function $A : [0,T] \mapto C^2_b(T^1_1M)$, and $B$ be a tensor valued adapted process, $B : [0,T] \times \Omega \mapto C^1_b(T^1M)$, suppose that $\exists C > 0$, such that,
    \begin{align*}
        \|\nabla B\|_{g}, \|\nabla A\|_{g}, \left\|\nabla^2 A\right\|_{g} , \|B\|_{g},\|A\|_{g} < C   &&\forall t \in [0,T], \Prob\text{-a.s}.
    \end{align*}
    
    Then, for any $\mathcal{F}_0$-measurable initial condition $U_0 \in OM$, there exists a unique solution to the lifted equation \eqref{Eq:TOM},
    \begin{equation*}
        dU_t = A^\hrz_i(t,U_t)\circ dW^i_t + B^\hrz(t,U_t)dt \text{ on } T_{U_t}OM,
    \end{equation*}
    in the sense of Definition \ref{def:hStrongSol}, where $B^\hrz(t,\omega,u) = (B(t, \omega, \pi(u))^\hrz$, and $A^h_i(t,u) = (A(t,\pi(u))ue_i)^\hrz$, with $e_i$ the $i^{\text{th}}$ coordinate unit vector of $\dR^d$.
\end{theorem}

\begin{remark}
    There is no integrability assumption imposed on the initial condition for this type of solution. However, $A$ is required to be deterministic and smooth, a necessity arising from the Stratonovich formalism, as discussed in the previous section, to ensure a chart-independent definition of the stochastic differential equation.
\end{remark}
We first establish the following lemma, which provides an existence and uniqueness result for the process in local charts.
From this point onward, we assume the hypotheses of Theorem~\eqref{thm:ExistUniq}.

\begin{lemma}
    \label{lem:ExistChart}
    Let $K \geq 1$. Given a normal coordinate chart $\Lambda = (B_M(x,r),\phi)$ around $x$ of radius $0<r\leq i_M$, and denote respectively $a,b,g, \Gamma$  the components in the chart $\Lambda$ of the tensors $A$ and $B$, the metric tensor and the Christoffel symbol. 
    
    Then for all $ s \in [0,T]$, for any $\mathcal{F}_s$-measurable $(\xi(s),\zeta(s)) \in B_d(0,r)\times B_d(0,K)^d$. There exists an unique adapted continuous process $(\xi^i(t),\zeta^j_k(t))_{t\in [s,T]}$, with initial condition $(\xi(s),\zeta(s))$, solution of 
    \begin{align*}
        \tag{$\lambda E_{\Lambda}$}
        \label{eq:localChartSDE}
        \begin{cases}
            d\xi^i(t)  = (\lambda(\xi) a^i_l(t,\xi)\zeta^l_m)\circ dW^m_t + \lambda(\xi)b^i(t,\xi)dt,\\
            d\zeta^k_m(t) = (-\lambda(\xi)\Gamma^k_{ij}\zeta^i_m a^j_l\zeta^l_n)\circ dW^n_t + (-\lambda(\xi)\Gamma^k_{ij}b^i \zeta^j_m)dt,
        \end{cases}
    \end{align*}
    where $\lambda$ is a bump function confining the process in $B_d(0,r)\times B_d(0,K^2)^d$.
    Moreover, the following property holds,
    \begin{equation*}
        \label{eq:zetanorm}
        \tag{$Z_\Lambda$}
        \zeta^i_k(t)g_{ij}(\xi_t)\zeta^j_m(t) \equiv \zeta^i_k(s)g_{ij}(\xi_s)\zeta^j_m(s) \ \ \ \forall t \in [s,T].
    \end{equation*}
    Furthermore, for any $p\geq 1$, there exists a constant $C_{T,p}>0$ depending on the geometry and the bounds on the tensors $A,B$, such that,
    \begin{equation}
        \label{est:LpestimChart}
        \dE \left[ \sup_{u\in [s,t]} \|\xi_u -\xi_s\|^p \Big| \mathcal{F}_s \right] \leq C_{T,p} |t-s|^\frac{p}{2}.
    \end{equation}
    Finally, for any $\epsilon>0$, there exists $C_{T,\epsilon}>0$, such that,
    \begin{equation*}
        \label{eq:exitestimate}
        \tag{$\Prob_{\tau} $}
        \Prob\left( \sup_{[s,t]} \| \xi_u - \xi_s \| > \epsilon | \mathcal{F}_s \right) \leq C_{T,\epsilon}|t-s|^2.
    \end{equation*}
\end{lemma}
\begin{proof}
    Let $\lambda : \dR^d \mapto [0,1]$ be the following $C^2_b(\dR^d)\cap W^{3,\infty}(\dR^d)$ bump function,
    \begin{equation*}
        \lambda(x) =    \begin{cases}
                            \hspace{2em}1 \hspace{5em}\text{ if }  \|x\| < \frac{2r}{3},\\
                            -6\left(\frac{3\|x\|-2r}{r}\right)^5+15\left(\frac{3\|x\|-2r}{r}\right)^4-10\left(\frac{3\|x\|-2r}{r}\right)^3+1 & \text{ if } \|x\| \in \left[\frac{2r}{3},r\right],\\
                            \hspace{2em}0 \hspace{5em}\text{ otherwise.}
                        \end{cases}
    \end{equation*}
    Introduce, the Ito Stochastic differential equation associated to the Stratonovich equation \eqref{eq:localChartSDE}.
    For the sake of conciseness, we note $a^i_j$ instead of $\lambda(\xi_t)a^i_j(t,\xi)$ and similarly $b^i$ for $\lambda(\xi)b^i(t,\xi)$. We omit $t$ in the differential equation at time t, and $\xi_t$ in the Christoffel symbol, so when we note $\partial_na^j_l$, we mean $\partial_n(\lambda(\xi_t)a^i_j(t,\xi))$. The Itô's counterpart of equation \eqref{eq:localChartSDE} is expressed as,
    \begin{align*}
        d\xi^i  = & (a^i_l \zeta^l_m)dW^m_t +\frac{1}{2}(\partial_j a^i_l a^j_k \zeta^k_m \zeta^l_m - a^i_l \zeta^n_m\zeta^j_m a^k_n \Gamma^l_{kj}) dt+ b^idt,\\
        d\zeta^k_m = & (-\Gamma^k_{ij}\zeta^i_m \zeta^l_q a^j_l)dW^q_t -(\Gamma^k_{ij}b^i \zeta^j_m)dt \\
                    & - \frac{1}{2}(\Gamma^k_{ij}\zeta^i_m \partial_n a^j_l \zeta^l_q a^n_p \zeta^p_q + \partial_n \Gamma^k_{ij}\zeta^i_ma^j_l\zeta^j_qa^n_p\zeta^p_q - \Gamma^k_{ij} \Gamma^i_{np}\zeta^n_m a^p_{l'}\zeta^{l'}_q a^j_l \zeta^l_q 
                     - \Gamma^k_{ij}\zeta^i_ma^j_l\Gamma^l_{np}\zeta^n_q\zeta^{l'}_qa^p_{l'})dt.
    \end{align*}

    Since the coefficients and the Christoffel symbol are regular in the chart, if we introduce the following stopping time,
    \begin{equation*}
        \tau = \inf \left\{t > s, \max^d_{m=1} \|\zeta_m(t)\| > 2K^2 \right\},
    \end{equation*}
    
    using classical existence and uniqueness result for stochastic differential equations with random coefficients, e.g \cite[Theorem 1.C p.90]{elworthy1982stochastic}, we obtain existence and uniqueness of the localized stochastic equation, since all the coefficients involved are $\Prob$-a.s Lipschitz and bounded up to the killing time. Now by proving the first property \eqref{eq:zetanorm} of the lemma, for the localized system, this will imply that the localized system stays in a ball inside the localization, since from the uniform equivalence of the norm in the chart, we obtain that,
    \begin{equation*}
        |\zeta_k(t)|^2 \leq K \zeta^i_k g^{ij} \zeta^j_k(t) = K \zeta^i_k g^{ij} \zeta^j_k(s) \leq K^2 |\zeta_k(s)|^2 \leq K^4.
    \end{equation*}
    Thus the localized solution is a solution to the non-localized equation.
    
    We now prove \eqref{eq:zetanorm}. Applying Itô-Stratonovich equivalence \cite[Prop.2.21,p.295]{karatzas2012brownian}, we obtain that $(\xi,\zeta)$ is a solution to the localized version of the Stratonovich equation \eqref{eq:localChartSDE}. Again we omit for conciseness $t,\xi(t)$ when evaluating a matrix value function or a coordinate in time, and the bump function $\lambda$.
    \begin{align*}
        d(\zeta^i_m(t) g_{ij}(\xi)\zeta^j_n(t)) = & g_{ij}\zeta^j_n \circ d\zeta^i_m + g_{ij}\zeta^i_m \circ d\zeta^j_n + \zeta^j_n \zeta^i_m \circ dg_{ij}\\
                                                = & (\partial_qg_{ij}\zeta^i_m\zeta^j_na^q_l\zeta^l_k -g_{ij}\zeta^j_n\Gamma^i_{pq}\zeta^p_ma^q_l\zeta^l_k - g_{ij}\zeta^i_m\Gamma^j_{qp}\zeta^p_na^q_l\zeta^l_k)\circ dW^k_t \\
                                                & + ( \partial_qg_{ij}\zeta^i_m\zeta^j_nb^q -\Gamma^i_{pq}b^q\zeta^i_m\zeta^p_n g_{ij} -\Gamma^i_{pq}b^q\zeta^p_m\zeta^j_n g_{ij}) dt\\
                                                = & (\zeta^i_m\zeta^j_na^q_l\zeta^l_k(\partial_q g_{ij} - g_{il}\Gamma^l_{jq} - g_{lj}\Gamma^l_{iq}))\circ dW^k_t \\
                                                & + (\zeta^i_m\zeta^j_nb^q(\partial_q g_{ij} - g_{il}\Gamma^l_{jq} - g_{lj}\Gamma^l_{iq})) dt
    \end{align*}
    Now by the following identity, $\partial_k g_{ij} - g_{il}\Gamma^l_{jk} - g_{lj}\Gamma^l_{ik} \eqcoord \nabla g = 0,$ e.g \cite{lee2006riemannian}, we obtain that, 
    \begin{equation*}
        d(\zeta^i_m(t) g_{ij}(\xi)\zeta^j_n(t)) \equiv 0.
    \end{equation*}
    Thus,
    \begin{equation*}
        \zeta^i_m(t) g_{ij}(\xi(t))\zeta^j_n(t) \equiv \zeta^i_m(s) g_{ij}(\xi(s))\zeta^j_n(s) \ \ \ \ \forall t \in [s,T].
    \end{equation*}

    We now prove the integrability estimate~\eqref{est:LpestimChart}. To this purpose, we denote by $M(s,\cdot)$ the matingale,
    \begin{equation*}
        M^i(s,u) \defeq \int^u_sa^i_l \zeta^l_mdW^m_v.
    \end{equation*}
    
    Then applying Burkholder-Davis-Gundy inequality, for $p\geq 1$,
    \begin{align*}
        \dE \left[\sup_{u\in [s,t]}\|\xi_u - \xi_s\|^p \Big|\mathcal{F}_s\right] &\leq C_p \dE\left[\left(\int_s^t\|\partial_j a^i_l a^j_k \zeta^k_m \zeta^l_m - a^i_l \zeta^n_m\zeta^j_m a^k_n \Gamma^l_{kj}\| + \|b\| du\right)^p\Big|\mathcal{F}_s \right] \\
        &\hspace{3em}+ \dE \left[\sup_{u\in [s,t]}\|M(s,u)\|^p\Big|\mathcal{F}_s\right],\\
        & \leq C_M K^2 (\|A\|^p_g\|\nabla A\|^p_g + \| A\|^{2p}_g + \|B\|^p_g)|t-s|^p\\
        &\hspace{3em}+ C_{BDG}\dE \left[\langle M(s,\cdot)\rangle^{p/2}_{t}\Big|\mathcal{F}_s\right],\\
        & \leq C_M K^2 C_{A,B} |t-s|^p+ C_{BDG}C_M K^2\|A\|^p|t-s|^{p/2}\\
        &\leq  C_{T,p} |t-s|^{p/2},
    \end{align*}

    where $C_{BDG}$ is the Burkholder-Davis-Gundy constant, and $C_{T,p}>0$ is a constant depending on a maximal time $T$, the tensors $A,B$ and the geometry.
    
    The estimate \eqref{eq:exitestimate} follows from Markov conditional inequality and the previous estimate with $p = 4$.
\end{proof}

We now prove the main theorem. We employ Itô's random switching construction~\cite{ito1950}, which we adapt to the moving frame solution.

\begin{proof}[Proof of Theorem \ref{thm:ExistUniq}]
    We start by constructing a process using a countable atlas of uniform normal coordinate systems. Since $(M,g)$ is of bounded geometry, from Lemma~\cite[Lemma 2.26 p.48]{aubin2012nonlinear} and the second-countability, there exists $0<r<i(M)$ sufficiently small such that, there exist an at most countable, uniformly locally finite, cover of $M$,
    \begin{equation*}
        \left\{O^\beta_0 = B_M\left(x^\beta, r/3\right), \beta \in \mathfrak{J} \right\}.
    \end{equation*}
    We define $O^\beta = B_M\left(x^\beta, r\right)$ and we fix $\phi^\beta$ an associated normal coordinate system around $x^\beta$, such that $\{\Lambda^\beta = (O^\beta,\phi^\beta), \beta \in \mathfrak{J}\}$ forms an atlas of $M$.
    
    Introduce the following partition of unity,
    \begin{equation*}
        \Tilde{O}^i = O^i_0 \big/ \left(\bigcup_{k = 0}^{i-1} \Tilde{O}^k\right).    
    \end{equation*}
    Fix $m\in \mathbb{N}^*$, let $t^m_k = \frac{kT}{m}$, for $k = 0,\cdots, m$, and for any chart $(O^\beta,\phi^\beta)$ with $\beta \in \mathfrak{J}$, define the $\mathcal{F}_0$-measurable random variables,
    \begin{equation*}
        {(\xi_0, \zeta_0)}^{(\beta, U_0)} = \begin{cases}
                                                (\phi^\beta(\pi(U_0)),d\phi^\beta(\pi(U_0))U_0\partial/\partial x^j) \text{ if } U_0 \in \Tilde{O}^\beta \\
                                                (0_{\dR^d},0_{\dR^{d^2}}) \text{ otherwise.}
                                             \end{cases}
    \end{equation*}
    From Lemma~\ref{lem:ExistChart}, for any normal chart $\Lambda^\beta$ there exists a unique solution to \eqref{eq:localChartSDE}, starting from $t^m_0$, with initial condition ${(\xi_0, \zeta_0)}^{(\beta, U_0)}$, we call this process $(\xi_\cdot, \zeta_\cdot)^{(\beta, t^m_0,U_0)}_{[0,t^m_1]}$. From property \eqref{eq:zetanorm}, on the event $\{\pi(U_0) \in \Tilde{O}^\beta \}$, 
    \begin{equation*}
        \zeta^i_m(t) g_{ij}(\xi(t)) \zeta^j_k(t) \equiv \zeta^i_l(t_0) g_{ij}(\xi(t_0)) \zeta^j_k (t_0) = \delta_{lk}.
    \end{equation*}
    The last equality holds from the construction of $\zeta_0$. As a result, $\zeta$ remains an orthonormal basis of $T_{\psi^\beta(\xi)}M$, and that $(\xi_\cdot, \zeta_\cdot)^{(\beta, t^m_0,U_0)}_{[0,t^m_1]}$ stays identifiable as an element of $OM$.
    
    Now set,
    \begin{equation*}
        U^{(m)}_\cdot = \left(\psi^\beta\left({\xi_\cdot}^{(\beta,t^m_0,U_0)}\right), d_\xi\psi^\beta{\zeta_\cdot}^{(\beta,t^m_0,U_0)}\right) \in OM  \text{ on } [t^m_0,t^m_1]\text{ if } \pi(U_0) \in \Tilde{O}^\beta.
    \end{equation*}
    Since $\mathfrak{J}$ is countable, $\left(U^{(m)}_\cdot\right)_{[t^m_0,t^m_1]}$ is an adapted process defined up to a $\Prob$-nullset.
    By continuing this procedure, define at step $k = 1, \cdots, m-1$, for any chart $\Lambda^\beta$ the $\mathcal{F}_{t^m_k}$-measurable random variable,
    \begin{equation*}
        {(\xi_{t^m_k}, \zeta_{t^m_k})}^{\left(\beta,U^{(m)}_{t^m_k}\right)} =    \begin{cases}
                                                                \left(\phi^\beta\left(\pi\left(U^{(m)}_{t^m_k} \right)\right),d\phi^\beta\left(\pi\left(U^{(m)}_{t^m_k} \right)\right)U^{(m)}_{t^m_k} \partial/\partial x^j\right) \text{ if } U^{(m)}_{t^m_k} \in \Tilde{O}^\beta \\
                                                                (0,0) \text{ otherwise.}
                                                            \end{cases}
    \end{equation*}
    From Lemma~\ref{lem:ExistChart}, for any chart $\Lambda^\beta$, note by ${(\xi_\cdot, \zeta_\cdot)}^{\left(\beta,U^{(m)}_{t^m_k}\right)}_{[t_k,t^m_{k+1}]}$ the unique solution of \eqref{eq:localChartSDE} starting at $t^m_k$ with initial condition ${(\xi_{t^m_k}, \zeta_{t^m_k})}^{\left(\beta,U^{(m)}_{t^m_k}\right)}$, and extend the process $U^{(m)}$ as,
    \begin{equation*}
        U^{(m)}_\cdot = \left(\psi^\beta\left({\xi_\cdot}^{\left(\beta,t^m_k,U^{(m)}_{t^m_k}\right)}\right), d_\xi\psi^\beta{\zeta_\cdot}^{\left(\alpha,t^m_k,U^{(m)}_{t^m_k}\right)}\right) \in OM  \text{ on } [t^m_k,t^m_{k+1}]\text{ if } \pi(U_{t^m_k}) \in \Tilde{O}^\beta.
    \end{equation*}
    For $m \in \mathbb{N}^*, k = 0,\cdots,m-1$, define the following events,
    \begin{align*}
        \Omega^{(k,m)}_{\alpha \beta} & = \left\{\omega\in\Omega, \pi\left(U^{(m)}_{(t^m_k,\omega)}\right)\in \Tilde{O}^\alpha , \pi\left(U^{(m)}_{(t^m_{k+1},\omega)}\right)\in \Tilde{O}^\beta, \xi^{\left(\alpha,t^m_k,U^{(m)}_{t^m_k}\right)}_{([t^m_k,t^m_{k+1}])} \in B_{\dR^d}\left(0,\frac{2r}{3}\right) \ \right\}\\
        \Omega^{(m)} & = \bigsqcup_{n_0,\cdots, n_m \in \mathfrak{K}^m} \bigcap_{k = 0,\cdots, m-1} \Omega^{(k,m)}_{n_k,n_{k+1}}
    \end{align*}
    Introduce the adapted process $(U_\cdot)_{[0,T]}$, defined up to a $\Prob$-nullset on the event $\bigcup_{m\in \mathbb{N}}\Omega^{(m)}$ as,
    \begin{equation*}
        U_\cdot = \begin{cases}
                U_\cdot^{(1)} \text{ on } \Omega^{(1)}, \\
                U_\cdot^{(m)} \text{ on } \Omega^{(m)}/ \left(\bigcup^{m-1}_{n=1} \Omega^{(n)}\right). \\
            \end{cases}
    \end{equation*}
    We now prove that $U$ is defined $\Prob$-a.s, by proving that, $\lim_m \Prob\left(\Omega^{(m)}\right) = 1$. To this end, define $\Tilde{\Omega}^{(k,m)}_{\beta} \defeq \Big\{ X_{t^m_k} \in \Tilde{O}^{\beta}, \sup_{[t^m_k,t^m_{k+1}]}\|\xi^{(\beta,k,m)}_{(u)}\| < \frac{2r}{3}\Big\}$, where $X_t \defeq \pi(U_t)$. Note that, 
    \begin{equation*}
        \left\{X_{t^m_k} \in \Tilde{O}^{\beta}, \sup_{[t^m_k,t^m_{k+1}]}\left\|\xi^{(\beta,k,m)}_{(u)} - \xi^{(\beta,k,m)}_{(t^m_k)}\right\| < \frac{r}{3}\right\} \subset \Tilde{\Omega}^{(k,m)}_{\beta},    
    \end{equation*}
     so that, for any event $A \in \mathcal{F}_{t^m_k}$,

    \begin{align*}
        \Prob\Big(A \cap \Tilde{\Omega}^{(k,m)}_{\beta}\Big) \geq & \Prob\left(A \cap \left\{ X_{t^m_k} \in \Tilde{O}^{\beta} , \sup_{[t^m_k,t^m_{k+1}]}\left\|\xi^{(\beta,k,m)}_{(u)} - \xi^{(\beta,k,m)}_{(t^m_k)}\right\| < \frac{r}{3} \right\}\right)\\
                                                                                     = & \int_{A \cap \left\{ X_{t^m_k} \in \Tilde{O}^{\beta}\right\}}\Prob\left(\sup_{[t^m_k,t^m_{k+1}]}\left\|\xi^{(\beta,k,m)}_{(u)} - \xi^{(\beta,k,m)}_{(t^m_k)}\right\| < \frac{r}{3} \Big| \mathcal{F}_{t^m_k}\right)\Prob(d\omega)\\
                                                                                     \geq & \int_{A \cap \left\{ X_{t^m_k} \in \Tilde{O}^{\beta}\right\}}\left(1-\frac{C}{m^2}\right)\Prob(d\omega) = \Prob\left(A \cap \{ X_{t^m_k} \in \Tilde{O}^{\beta}\}\right)\left(1-\frac{C}{m^2}\right),
    \end{align*}
    where $C$ is a constant, independent of $\beta\in\mathfrak{J}$, depending on the tensor norms of $A$ and $B$, and the geometry.
    Thus we conclude,
    \begin{align}
        \Prob\left(\Omega^{(m)}\right) & = \sum_{n_0,\cdots,n_{m-1}} \Prob\left(\Omega^{(1,m)}_{(n_1,n_2)}\cap \cdots \cap \Omega^{(m-2,m)}_{(n_{m-2},n_{m-1})} \cap \Tilde{\Omega}^{(k,m)}_{m-1}\right) \nonumber \\
        &\geq \left(1-\frac{C}{m^2}\right)\sum_{n_0,\cdots,n_{m-1}} \Prob\left(\Omega^{(1,m)}_{(n_1,n_2)}\cap \cdots \cap \Omega^{(m-2,m)}_{(n_{m-2},n_{m-1})}\right) \nonumber \\
        & \geq \left(1-\frac{C}{m^2}\right)^m \xrightarrow[]{m \mapto +\infty } 1.
        \label{eq:convergenceOmegm}
    \end{align}
    
    We now verify that the process is a solution in the sense of Definition~\ref{def:hStrongSol}. Let $\{(O^\alpha,\phi^\alpha), \alpha \in \mathfrak{K}\}$ be a regular localization of the dynamical system. Let $t \in [0,T], \alpha \in \mathfrak{K}$, and introduce as in Definition \ref{def:hStrongSol},
    \begin{equation*}
        \Omega^{\alpha}_t \defeq \{ \omega \in \Omega, \pi(U_t(\omega)) \in O^\alpha\} \in \mathcal{F}_t,
    \end{equation*}
    and the exit time $\tau^{\alpha}_t : \Omega^{\alpha}_t \mapto [t,T]$ of $O^\alpha$.
    From the convergence of $\Omega^{(m)}$ in \eqref{eq:convergenceOmegm},
    
    \begin{equation*}
        \Omega^{\alpha}_t = \left(\bigcup_{m\in \mathbb{N}} \Omega^{\alpha}_t \cap \Omega^{(m)} \right)\cup \Omega_0,
    \end{equation*}
    
    where $\Omega_0$ is a $\Prob$-nullset. 
    Denote by $(\Bar{\xi}^i, \Bar{\zeta}^j_k)_{[t,\tau^\alpha_t)}$ the process defined in $\Omega_t^\alpha$ by,
    \begin{equation*}
        (\Bar{\xi}^i, \Bar{\zeta}^j_k) = \left(\phi^\alpha\left(\pi\left(U\right)\right),d\phi^\alpha\left(\pi\left(U\right)\right)U \partial/\partial x^j\right).
    \end{equation*}
    
    Take a multi-index $\overset{\rightharpoonup}{n} \in \mathfrak{J}^m$ and define
    \begin{equation*}
        \Omega^{(m)}_{\overset{\rightharpoonup}{n}} = \bigcap_{k = 0,\cdots, m-1} \Omega^{(k,m)}_{n_k,n_{k+1}}.
    \end{equation*}

    Then, either $\Omega^{\alpha}_t\bigcap  \Omega^{(m)}_{\overset{\rightharpoonup}{n}} = \varnothing$, in which case the property is verified, or $\Omega^{\alpha}_t\bigcap  \Omega^{(m)}_{\overset{\rightharpoonup}{n}} \neq \varnothing$. 
    In the latter case, by the construction of $U$, from the change of coordinates rule, $\forall s \in [t,\tau^\alpha_t)$, if  $\frac{lT}{m} \leq s < \frac{(l+1)T}{m}$ then,
    \begin{equation*}
        \left(\Bar{\xi}(s), \Bar{\zeta}_k(s)\right) = \left((\varphi^\alpha \circ \psi^{n_l})\left(\xi^{(n_l,t_l, U^{(m)}_{t_k})}_{(s)}\right), \partial_j(\varphi^\alpha \circ \psi^{n_l})\left(\zeta^{(n_l,t_l, U^{(m)}_{t_k})}_{k(s)}\right)^j\right) 
    \end{equation*}
    For the sake of clarity, we introduce the following notations for the up coming computations. First $\Bar{x}(y) \defeq (\varphi^\alpha \circ \psi^\beta) (y), \frac{\partial \Bar{x}^i}{\partial x^k}(y) = \partial_k(\varphi^\alpha \circ \psi^\beta (y))^i$, respectively, $x(\Bar{y}) \defeq (\varphi^\beta \circ \psi^\alpha) (\Bar{y}), \frac{\partial x^i}{\partial \Bar{x}^k}(\Bar{y}) = \partial_k(\varphi^\beta \circ \psi^\alpha (\Bar{y}))^i$, then by the change of coordinates rule, $(\Bar{\xi}^i, \Bar{\zeta}^j_k) = (\Bar{x}^i(\xi), \frac{\partial\Bar{x}^j}{\partial x^l}(\xi)\zeta^l_k)$.
    From the construction of the process $U$, on the considered event, $\xi$ stays where the bump function is equal to $1$. From the Stratonovich chain rule, we obtain the following equation on $(\Bar{\xi}^i, \Bar{\zeta}^j_k)$, again, we omit $t,\xi$ for the matrix-valued functions.
    \begin{align*}
            d\Bar{\xi}^i & = \left(\frac{\partial\Bar{x}^i}{\partial x^l} a^l_m \zeta^m_k\right)\circ dW^k_t + \frac{\partial\Bar{x}^i}{\partial x^l} b^l dt,\\
            d\Bar{\zeta}^j_k & =  \left(\frac{\partial^2 \Bar{x}^j}{\partial x^l \partial x^i} a^i_m\zeta^m_n\zeta^l_k - \frac{\partial \Bar{x}^j}{\partial x^p}\Gamma^p_{il}a^i_m \zeta^m_n \zeta^l_k\right)\circ dW^n_t + \left(\frac{\partial^2 \Bar{x}^j}{\partial x^l \partial x^i} b^i \zeta^l_k - \frac{\partial \Bar{x}^j}{\partial x^p}\Gamma^p_{il}b^i\zeta^l_k\right)dt.
    \end{align*}
    We recall that,
    \begin{equation*}
        \frac{\partial \Bar{x}^j}{\partial x^l} \frac{\partial x^l}{\partial \Bar{x}^p}= \partial_p(\varphi^\alpha \circ  \psi^\beta(\varphi^\beta \circ \psi^\alpha))^j = \delta^j_p,
    \end{equation*}
    and by the change of coordinates of the tensors,
    \begin{equation*}
        \Bar{a}^i_j = \frac{\partial\Bar{x}^i}{\partial x^l}\frac{\partial x^k}{\partial \Bar{x}^j}a^l_k,
    \end{equation*} where $\Bar{a}$(resp. $a$) are the coefficients of $A$ the $(1,1)$-tensor in the chart $\Lambda^\alpha$(resp. $\Lambda^\beta$).
    We then obtain,
    \begin{align*}
            d\Bar{\xi}^i & = (\Bar{a}^i_m \bar{\zeta}^m_k)\circ dW^k_t + \bar{b}^i dt,\\
            d\Bar{\zeta}^j_k & =  \left(\frac{\partial^2 \Bar{x}^j}{\partial x^l \partial x^i} \frac{\partial x^i}{\partial \Bar{x}^q}\Bar{a}^q_m\Bar{\zeta}^m_n\frac{\partial x^l}{\partial \Bar{x}^p}\Bar{\zeta}^p_k - \frac{\partial \Bar{x}^j}{\partial x^m}\Gamma^m_{il}\frac{\partial x^i}{\partial \Bar{x}^q}\Bar{a}^q_{m'} \Bar{\zeta}^{m'}_n \frac{\partial x^l}{\partial \Bar{x}^p}\Bar{\zeta}^p_k\right)\circ dW^n_t
            \\ & \hspace{2em} + \left(\frac{\partial^2 \Bar{x}^j}{\partial x^l \partial x^i}\frac{\partial x^i}{\partial \Bar{x}^q} \Bar{b}^q \frac{\partial x^l}{\partial \Bar{x}^p}\Bar{\zeta}^p_k - \frac{\partial \Bar{x}^j}{\partial x^m}\Gamma^m_{il}\frac{\partial x^i}{\partial \Bar{x}^q}\Bar{b}^q\frac{\partial x^l}{\partial \Bar{x}^p}\Bar{\zeta}^p_k\right)dt.
    \end{align*}
    Rewriting the second equation as,
    \begin{align*}
        d\Bar{\zeta}^j_k = &  \left(\left(\frac{\partial^2 \Bar{x}^j}{\partial x^l \partial x^i} \frac{\partial x^i}{\partial \Bar{x}^q}\frac{\partial x^l}{\partial \Bar{x}^p} - \frac{\partial \Bar{x}^j}{\partial x^m}\Gamma^m_{il}\frac{\partial x^i}{\partial \Bar{x}^q}\frac{\partial x^l}{\partial \Bar{x}^p}\right)\Bar{a}^q_{m'} \Bar{\zeta}^{m'}_n \Bar{\zeta}^p_k\right)\circ dW^n_t \\ 
       & + \left(\left(\frac{\partial^2 \Bar{x}^j}{\partial x^l \partial x^i}\frac{\partial x^i}{\partial \Bar{x}^q}\frac{\partial x^l}{\partial \Bar{x}^p} - \frac{\partial \Bar{x}^j}{\partial x^m}\Gamma^m_{il}\frac{\partial x^i}{\partial \Bar{x}^q}\frac{\partial x^l}{\partial \Bar{x}^p}\right)\Bar{b}^q \Bar{\zeta}^p_k \right)dt.
    \end{align*}
    Combining the identity,
    \begin{equation*}
        \frac{\partial^2 \Bar{x}^j}{\partial x^l \partial x^i}\frac{\partial x^i}{\partial \Bar{x}^q}\frac{\partial x^l}{\partial \Bar{x}^p} + \frac{\partial^2 x^l}{\partial \Bar{x}^p \partial \Bar{x}^q}\frac{\partial \Bar{x}^j}{\partial x^l} =  \partial_{pq}(\varphi^\alpha \circ  \psi^\beta(\varphi^\beta \circ \psi^\alpha))^j= 0,
    \end{equation*} with the change of coordinates of the Christoffel symbol identity,
    \begin{equation*}
        \Bar{\Gamma}^j_{pq} = \frac{\partial \Bar{x}^j}{\partial x^m}\Gamma^m_{il}\frac{\partial x^i}{\partial \Bar{x}^q}\frac{\partial x^l}{\partial \Bar{x}^p} + \frac{\partial^2 x^l}{\partial \Bar{x}^p \partial \Bar{x}^q}\frac{\partial \Bar{x}^j}{\partial x^l},
    \end{equation*}
    we conclude that,
    \begin{equation*}
        -\Bar{\Gamma}^j_{pq} = \left(\frac{\partial^2 \Bar{x}^j}{\partial x^l \partial x^i}\frac{\partial x^i}{\partial \Bar{x}^q}\frac{\partial x^l}{\partial \Bar{x}^p} - \frac{\partial \Bar{x}^j}{\partial x^m}\Gamma^m_{il}\frac{\partial x^i}{\partial \Bar{x}^q}\frac{\partial x^l}{\partial \Bar{x}^p}\right).
    \end{equation*}
    Thus, $(U_\cdot)$ is a solution of \eqref{Eq:TOM} in the sense of Definition \ref{def:hStrongSol}.\\

    We conclude the proof, with the proof of the uniqueness. Let $(U^1_\cdot),(U^2_\cdot)$ be two solutions of \eqref{Eq:TOM}, with initial condition equal $\Prob$-a.s, we can suppose that they are solutions on the same uniformly regular atlas. Choose a countable dense subset $\mathcal{D}$ of $[0,T]$ with $0 \in \mathcal{D}$. For $t_0 \in \mathcal{D}$, define the events $E^\alpha_{t_0} = \{U^1_{t_0} = U^2_{t_0} \in O^\alpha_0\}$ and  $\Omega^\alpha_{[t_0,t_1]} = \{E^\alpha_{t_0}, t_1 < \tau_{t_0}^\alpha\}$, where $\tau_{t_0}^\alpha: E^\alpha_{t_0} \mapto [t,T]$, is the minimum of the exit time of $O^\alpha_0$ for each of the processes. Then by hypothesis, since they are both solutions of \eqref{Eq:TOM}, up to a $\Prob$-nullset, that we note $Z^\alpha_{[t_0,t_1]}$, they are solutions in the chart $\Lambda^\alpha$ of the same Euclidean stochastic differential equation, and from the uniqueness result of Lemma \ref{lem:ExistChart}, both processes are equal on $\Omega^\alpha_{[t_0,t_1]} \big/ Z^\alpha_{[t_0,t_1]}$. \\
    
    Taking $\omega \in \Omega$, both process are continuous a.s,
    let $s \in [0,T]$ and assume that the processes agree on $(\omega, [0,s))$, then there exists $\alpha \in \mathfrak{K}$, with $\omega \in \Omega^\alpha_{[t_0,t_1]}$ for some $t_0,t_1 \in \mathcal{D}$ satisfying $0\leq t_0\leq s<t_1\leq T$.
    Then if $\omega \notin Z^\alpha_{[t_0,t_1]}$ we have that the processes agree on $(\omega, [0,t_1])$. Noting that $t_1$ is strictly greater than $s$, using the sample continuity and the fact that $\bigcup_{t_0,t_1\in \mathcal{D}, \alpha \in \mathfrak{K}} Z^\alpha_{[t_0,t_1]}$ is a $\Prob$-nullset, we obtain the desired result.
\end{proof}

In~\cite[Lemma 9B, p.146]{elworthy1982stochastic}, a global Itô's formula is given for general manifolds and general lifted solutions to stochastic differential equations with deterministic coefficients, that hold in our case. We here prove the case of stochastic drift, given the regularity of the manifold and the coefficients.

\begin{proposition}
    \label{prop:ItoFormula}
    Let $(M,g)$ be of bounded geometry, and suppose that $B$ is bounded $\Prob$-almost surely and $A,\nabla A$ are bounded. Let $(U_t)_t$ be the associated moving frame strong solution as Definition~\ref{def:hStrongSol}.
    Then $(X_t \defeq \pi(U_t))_t$ is a continuous adapted process taking values in $M$ such that, $\forall \varphi \in C^{1,2}([0,T]\times M, \dR),$ with bounded derivatives, $\forall t\geq s \in [0,T]$, $(M^s_t)_{t\geq s}$ defined by,
    \begin{align*}
        M^s_t \defeq \varphi(t, X_t) - \varphi(s, X_s)  - \int_s^t\left(\partial_t \varphi + \left(B + \frac{1}{2}\nabla A \cdot A\right)\cdot \nabla \varphi + \frac{1}{2}\Sigma \cdot \nabla^2\varphi\right)(u,X_u) du,
    \end{align*}
    is a martingale, and it holds that,
    \begin{equation*}
        \varphi(t, X_t) - \varphi(s, X_s)  = \int_s^t(\partial_t \varphi + B \cdot \nabla \varphi )(u,X_u) du + \int_s^t(A_i(u,U_u) \cdot \nabla \varphi) \circ dW^i_u.
    \end{equation*}
    or in Ito's form,
    \begin{align*}
          \varphi(t, X_t) - \varphi(s, X_s) & = \int_s^t\left(\partial_t \varphi + \left(B + \frac{1}{2}\nabla A \cdot A\right)\cdot \nabla \varphi + \frac{1}{2}\Sigma \cdot \nabla^2\varphi \right)(u,X_u) du \\ 
          & \hspace{1.5em}  + \int_s^t(A_i(u,U_u) \cdot \nabla \varphi)dW^i_u,
    \end{align*}
    where $\Sigma$ is a $(2,0)$-tensor, defined as $\Sigma = A\cdot A^*$, in local coordinates this can be written as, $\sigma^{ij} \eqcoord a^i_kg^{kl}a^j_l$.
\end{proposition}

\begin{proof}
    Using the decomposition,
    \begin{equation*}
         \Omega = \left(\bigcup_{m \in \mathbb{N} }\bigsqcup_{n_0,\cdots, n_m \in \mathfrak{K}^m} \bigcap_{k = 0,\cdots, m-1} \Omega^{(k,m)}_{n_k,n_{k+1}}\right) \cup \Omega_0,
    \end{equation*}
    from the proof of Theorem~\ref{thm:ExistUniq}, where $\Omega_0$ is a $\Prob$-null set. The proof follows directly from applying Itô's formula~\cite[Theorem 5.1 Chapiter II]{ikeda2014stochastic} in local charts.
\end{proof}
\subsection{Stochastic flow estimate and integrability}
The notion of solution considered here does not require an integrability assumption, in this section, we will prove a Lipschitz in time integrability estimate for the solutions.
In the Euclidean setting, this follows directly from the triangle inequality, Jensen's inequality, and the Burkholder-Davis-Gundy inequality. On a Riemannian manifold, however, one must apply Itô's formula to the distance function. A key difficulty arises outside the injectivity radius, where the distance function fails to be differentiable. For instance, on the one-dimensional torus $\mathbb{T} = \mathbb{R}/\mathbb{Z}$, the squared distance function $d^2(0,\cdot)$ is not differentiable at $1/2$.


\begin{lemma}[Proposition 26.49~\cite{chow2010ricci}]
    \label{lem:regulDist}
    Let $(M,g)$ be a manifold of bounded geometry. There exists a constant $C_M>0$ depending on the bound on the sectorial curvature and the dimension of $M$ such that, for any $x\in M$, there exists a $C^\infty$ function $\Tilde{r}_x:M \to \dR$ such that:
    \begin{equation*}
        |\Tilde{r}_x(y) - d(x,y)|< C_M, \ \ \ \forall y \in M.
    \end{equation*}
    and its derivatives are uniformly bounded in $x$,
    \begin{equation*}
        \|\nabla\Tilde{r}_x\|_g, \|\nabla^2 \Tilde{r}_x\|_g \leq C_M \ \ \ \  \forall x \in M
    \end{equation*}
\end{lemma}





\begin{theorem}
    Let $(M,g)$ be a Riemannian Manifold of bounded geometry.
    Let $\FiltProbSpBrwn$  be a filtered probability space with a $\Prob$-complete filtration, equipped with an adapted $d$-dimensional Brownian motion, and suppose that $(U_\cdot)_{[0,T]}$ is a moving frame solution of 
    \begin{equation*}
        dU_t = A_i^\hrz(U_t) \circ dW^i_t + B^\hrz(t,U_t)dt.
    \end{equation*}
    Then, we have the flow estimate,
    \begin{equation}
        \label{eq:ProbFlowEstimate}
        \dE[d^p(X_s,X_t)] < C_{T,p}|t-s|^{p/2} \ \ \ \forall p \geq 1,
    \end{equation}
    where $C$ depends on $T,p,\|B\|_\infty, \|A\|_\infty, \|\nabla A \|_\infty$ and geometrical bounds.

    Furthermore, if there exists $x_* \in M$ such that, $\dE[d^p(x^*,X_0)] < +\infty$, we write  $X_0 \in L^p(\Omega, d(x_*, \cdot)d\Prob)$.
    Then \eqref{eq:ProbFlowEstimate} ensures that the initial integrability is preserved.
\end{theorem}

\begin{proof}
    Let us introduce the following notations, $r_y(x) = d(x,y)$ and $R_{s,t} = d(X_s,X_t) = r_{X_s}(X_t)$ for any $s<t \in [0,T]$.
    Since $M$ has bounded geometry, according to Lemma \ref{lem:regulDist} there exists a regularization, in the second variable, of the distance function, that we denote by $\Tilde{d}: M\times M \mapto [0,\infty[$. Similarly we note by $\Tilde{r}_x(y) = \Tilde{d}(x,y)$, and $\Tilde{R}_{s,t} = \Tilde{d}(X_s,X_t)$.
    
    Now, start by noting that
    \begin{align*}
        d^p(X_s, X_t) &\leq d^p(X_s, X_t)\indic{\sup_{u \in [s,t]} d(X_s, X_u) <i(M)/2} + (\Tilde{d}(X_s, X_t)+C_M)^p \indic{\sup_{u \in [s,t]} d(X_s, X_u) \geq i(M)/2},\\
        &\leq d^p(X_s, X_t)\indic{\sup_{u \in [s,t]} d(X_s, X_u) < i(M)/2} +C_{p,M} \Tilde{d}^p(X_s, X_t) + C_{p,M}\indic{\sup_{u \in [s,t]} d(X_s, X_u) \geq i(M)/2},
    \end{align*}
    where $C_M>0$ is the constant of Lemma~\ref{lem:regulDist}, and $C_{p,M} >0$ is the constant obtained from the convexity of the $p$-th power and $C_M$.
    
    We will first show that $\dE[\Tilde{R}^p_{s,t}]$ is finite. Since $\Tilde{d}$ is regular, from Ito's formula of Proposition~\ref{prop:ItoFormula}, we obtain that, 
    \begin{align*}
        \Tilde{R}_{s,t} & = \int_s^t\left(\left(B +\frac{1}{2} \nabla A \cdot A\right) \cdot\nabla \Tilde{r}_{X_s} + \frac{1}{2} \Sigma \cdot \nabla^2\Tilde{r}_{X_s}\right)(u,X_u)du 
        \\ &\hspace{5em} + \int_s^t(\nabla \Tilde{r}_{X_s})(u, X_u) \cdot A_i(u,U_u) dW^i_u.
    \end{align*}

    From Lemma \ref{lem:regulDist}, and the bound on $B$, $A$ and $\nabla A$ the first integral term is bounded as,
    \begin{equation*}
        \left|\int_s^t\left(\left(B +\frac{1}{2} \nabla A \cdot A\right) \cdot\nabla \Tilde{r}_{X_s} +\frac{1}{2} \Sigma \cdot \nabla^2\Tilde{r}_{X_s}\right)(u,X_u)du\right|^p \leq C|t-s|^p,
    \end{equation*}
    for some constant $C>0$.
    Denote by $\Tilde{M}_{s,u}$, the martingale defined as,
    \begin{equation*}
        \Tilde{M}_{s,u} = \int_s^u(\nabla \Tilde{r}_{X_s})(X_v) \cdot A_i(u,U_u) dW^i_v.
    \end{equation*}
    
    Then, by Burkholder-Davis-Gundis inequality,
    \begin{equation*}
        \dE\left[\Tilde{d}^p(X_s,X_t)\right] \leq C\left(|t-s|^p + \dE\left[\sup_{u\in [s,t]} \Tilde{M}^p_{s,u}\right]\right) \leq C\left(|t-s|^p + \dE\left[\langle\Tilde{M}_{s}\rangle^{p/2}_t\right]\right), \\ 
    \end{equation*}
    for some constant $C>0$.
    
    Since,
    \begin{equation*}
        |\nabla \Tilde{r}_{X_s}(X_t) \cdot A_i(u,U_u)| \leq \|\nabla \Tilde{r}\|_g\|A\|_g,
    \end{equation*}
    We obtain that,
    \begin{equation}
        \label{eq:estimTildDist}
        \dE[\Tilde{d}^p(X_s,X_t)] \leq C\left(|t-s|^p + |t-s|^{p/2}\right).
    \end{equation}
    for a new constant $C>0$.

    
    
    On the event $\left\{\sup_{u \in [s,t]} d(X_s, X_u) < i(M)/2\right\}$, the solution stays in a normal coordinate chart around $X_s$, and coincide with the solution from Lemma~\ref{lem:ExistChart}, and noting $(\xi_{\cdot})_{[s,t]}$ the process in such a normal chart, one has on this event,
    \begin{equation*}
        \|\xi_u - \xi_s\| = d(X_s, X_u).
    \end{equation*}
    So that from estimate~\eqref{est:LpestimChart}, one obtains,
    \begin{equation*}
        \dE[d^p(X_s,X_t)\indic{\sup_{u \in [s,t]} d(X_s, X_u) < i(M)/2}] \leq C_{T,p}|t-s|^{p/2},
    \end{equation*}
    for some constant $C_{T,p}>0$, depending on $T>0$ and $p$.
    
    Finally, introducing the exit time $\tau = \inf\{u>s, d(X_s, X_u) > i(M) \} \bigwedge T$, one can also apply~\eqref{eq:estimProbexit} since it first needs to exist the ball $i(M)$ before $i(M)/2$, to use the estimate in the normal chart, and obtain,
    \begin{align}
        \label{eq:estimProbexit}
        \Prob\left(\sup_{u \in [s,t]} d(X_s, X_u) \geq i(M)/2\right) \leq C|t-s|^p\wedge1,
    \end{align}
    Summing-up the estimates \eqref{eq:estimTildDist} and \eqref{eq:estimProbexit}, in
    \begin{align*}
        \dE[d^p(X_s,X_t)]  & \leq \dE[d^p(X_s,X_t) \indic{\sup_{u \in [s,t]} d(X_s, X_u) < i(M)/2}] \\
        &\hspace{2em} + C_p\dE \left[\Tilde{d}^p(X_s,X_t)\right] + C_p \Prob\left(\sup_{u \in [s,t]} d(X_s, X_u) \geq i(M)/2\right),
    \end{align*}
    concludes the proof.    
\end{proof}




\begingroup
\renewcommand{\thesubsection}{\Alph{subsection}}
\renewcommand{\thetheorem}{\thesubsection.\arabic{theorem}}
\renewcommand{\thedefinition}{\thesubsection.\arabic{definition}}
\renewcommand{\thedefinitionproposition}{\thesubsection.\arabic{definitionproposition}}
\renewcommand{\theproposition}{\thesubsection.\arabic{proposition}}



\bibliographystyle{unsrt}
\bibliography{biblio}

\end{document}